\documentclass[reqno, 11pt, a4paper]{amsart} 

\usepackage[T5,T1]{fontenc}
\usepackage{mlmodern}

\usepackage[utf8]{inputenc}
\allowdisplaybreaks
\usepackage{amsfonts}
\usepackage{amsmath}
\usepackage{amssymb}
\usepackage{amsthm}
\usepackage[text={33pc,605pt},centering, margin=1.25in]{geometry}     

\usepackage{mathrsfs} 
\usepackage[dvipsnames]{xcolor}

\usepackage{bbm}
\usepackage{mathtools}

\usepackage{dsfont}
 
\usepackage[pagebackref=true,colorlinks=true, linkcolor=Blue, citecolor=Blue, pdfencoding=auto, psdextra]{hyperref}
\renewcommand*\backref[1]{\ifx#1\relax \else (Cited on #1) \fi}

\usepackage{appendix}
\usepackage{enumitem}
\usepackage{float}

\usepackage{tikz}
\usepackage{tikz-cd} 
\usetikzlibrary{arrows, arrows.meta}

\usepackage[nameinlink]{cleveref} 
\usepackage{scalerel}[2016/12/29]

\theoremstyle{plain}
\newtheorem{definition}{Definition}

\newtheorem{lemma}[definition]{Lemma}

\newtheorem{theorem}[definition]{Theorem}
\newtheorem{remark}[definition]{Remark}

\newtheorem{question}[definition]{Question}
\theoremstyle{definition}

\numberwithin{definition}{section}
\numberwithin{equation}{section}

\newcommand*{\R}{\mathbb{R}}
\newcommand*{\C}{\mathbb{C}}

\newcommand*{\Z}{\mathbb{Z}}

\newcommand*{\N}{\mathbb{N}}

\renewcommand*{\d}{\mathrm{d}}
\newcommand*{\e}{\mathrm{e}}

\makeatletter
\newcommand{\subalign}[1]{%
  \vcenter{%
    \Let@ \restore@math@cr \default@tag
    \baselineskip\fontdimen10 \scriptfont\tw@
    \advance\baselineskip\fontdimen12 \scriptfont\tw@
    \lineskip\thr@@\fontdimen8 \scriptfont\thr@@
    \lineskiplimit\lineskip
    \ialign{\hfil$\m@th\scriptstyle##$&$\m@th\scriptstyle{}##$\hfil\crcr
      #1\crcr
    }%
  }%
}
\makeatother

\crefname{equation}{}{}

\title[Spectral consequences of invariant measures]{Can one feel the existence of a non-trivial invariant measure?}

\author[Y. Steenbeck]{Yannic Steenbeck}

\address[Yannic Steenbeck]{TU Braunschweig, Institut für Mathematische Stochastik,
Germany.}
\email{yannic.steenbeck@tu-braunschweig.de}



\date{\today}

\usepackage{soul}

\begin{document}

\begin{abstract}
    It is shown that a bounded linear map on a complex separable Hilbert space with non-trivial invariant probability measures doesn't have to possess eigenvalues.
    This resolves a question indicated by Flytzanis in 1995 and concretely asked by Grivaux--López-Martínez from 2023 resp.\ Grivaux--Matheron--Menet from 2021.
    Still, as a positive result, we prove that every bounded linear operator on a separable complex Banach space for which a non-fixing invariant probability measure exists, has to have some approximate point spectrum on the unit circle minus \(1\).
\end{abstract}
\maketitle

\section{Introduction and main results}

It is of considerable mathematical interest to study invariant measures of various dynamics and relate their existence and behavior to the underlying operators, in particular their spectral properties. In particular the study of invariant (probability) measures of \emph{linear} operators is something that naturally appears in as diverse fields such as PDE, stochastics, dynamical systems and so on. 
Concretely, consider the setting of a separable complex Hilbert space \(\mathcal{H}\) and a bounded linear operator \(T \colon \mathcal{H} \to \mathcal{H}\). We want to understand a little bit the relation between the spectrum \(\sigma(T)\) of \(T\) and the existence of non-trivial \(T\)-invariant probability measures \(\varrho \neq \delta_0\) on \(\mathcal{H}\) (equipped with the Borel-\(\sigma\)-algebra). 

An easy construction of such a probability measure \(\varrho\) is possible if \(T\) has point-spectrum, i.e.\ eigenvalues, on the unit circle \(\mathbb{S}^1 := \{\lambda \in \C \,\colon\, \vert \lambda \vert = 1\}\), by sampling uniformly a phase \(\e^{i \theta}\) from \(\mathbb{S}^1\) and then multiplying a fixed eigenvector \(x \in \mathcal{H} \setminus\{0\}\) to an eigenvalue \(\alpha \in \mathbb{S}^1\) with it.
It is now a valid question, if the existence of a \(T\)-invariant non-trivial probability measure already implies that \(T\) has to possess eigenvalues (on the unit circle \(\mathbb{S}^1\)).
This question was indicated in \cite{Flytzanis1995} and later concretely posed as
\begin{question}[{\cite[Question 8.6]{GrivauxMartinez2023}}, {\cite[Question 8.3]{grivaux2021linear}}]\label{question:main_question}
    Is there a bounded linear map \(T \colon \mathcal{H} \to \mathcal{H}\) on some separable complex Hilbert space \((\mathcal{H}, \langle \cdot, \cdot\rangle_{\mathcal{H}})\) which has no eigenvalues but there exists a non-trivial \(T\)-invariant probability measure \(\varrho \neq \delta_0\) on \(\mathcal{H}\) (equipped with the Borel-\(\sigma\)-algebra)?
\end{question}

A first partial answer is given in the following
\begin{theorem}[{\cite[Theorem]{Flytzanis1995}} and {\cite[Lemma 4.4]{GrivauxMartinez2023}}]
    Let \(T\) be a bounded linear map \(T \colon \mathcal{H} \to \mathcal{H}\) on some separable complex Hilbert space \((\mathcal{H}, \langle \cdot, \cdot\rangle_{\mathcal{H}})\) and \(\varrho\) be a (non-trivial) \(T\)-invariant probability measure on \(\mathcal{H}\) fulfilling the moment condition \(\int \, \Vert x\Vert_{\mathcal{H}}^2 \, \varrho(dx) < \infty\). Then,
    \begin{align}
        \mathrm{supp}(\varrho)
        \subseteq \overline{\mathrm{span} \,\mathcal{E}(T)}^{\Vert\cdot\Vert_{\mathcal{H}}},
    \end{align} for the unimodular eigenvectors \(\mathcal{E}(T) = \{x \in \mathcal{H} \,\colon\, Tx = \lambda x \text{ for some } \lambda \in \mathbb{S}^1\}\).
\end{theorem}

The preceding assertion and especially its proof via comparison with a Gaussian measure, may provoke the thought that some moment condition is just a technical artifact here. This is indeed not the case and we can give an answer to Question~\ref{question:main_question} with the following theorem.
\begin{theorem}\label{theorem:example}
    There exists a separable complex Hilbert space \((\mathcal{H}, \langle \cdot, \cdot\rangle_{\mathcal{H}})\) which supports a bounded and boundedly invertible linear map \(T \colon \mathcal{H} \to \mathcal{H}\) which has no eigenvalues but a non-trivial \(T\)-invariant probability measure \(\varrho \neq \delta_0\) on \(\mathcal{H}\).
\end{theorem}
\begin{remark}
    Close to our initial motivation to study this problem, it should be even possible, using e.g.\ very similar ideas as in the construction presented here, to find a complex separable Hilbert space \((\mathcal{H}, \langle \cdot, \cdot\rangle_{\mathcal{H}})\) supporting an analytic \(C_0\)-semigroup \((T_t)_{t \geq 0}\) on it such that no \(T_t\), \(t > 0\), has any eigenvalues, together with a non-trivial probability measure \(\varrho \neq \delta_0\) such that \(\varrho\) is \(T_t\)-invariant for all \(t \geq 0\). 
\end{remark}

Still, the existence of a non-trivial invariant measure mirrors in the approximate point spectrum, and we can give the following positive result, which is even valid on Banach spaces.
\begin{theorem}\label{theorem:spectrum_on_circle_minus_1_from_nonfixing_invariant_law}
	Let \(T\) be a bounded linear map on some separable complex Banach space \(\mathcal{X}\).
    Suppose there is a non-trivial \(T\)-invariant probability measure \(\varrho \neq \delta_0\) on \(\mathcal{X}\) (equipped with the Borel-\(\sigma\)-algebra).
    Then, it holds for the approximate point spectrum \(\sigma_{\mathrm{ap}}(T)\) of \(T\) that
    \begin{align}
        \sigma_{\mathrm{ap}}(T) \cap \mathbb{S}^1 \neq \emptyset.
    \end{align}
    Furthermore, if \(\varrho\) is \emph{non-fixing}, i.e.\ if \(\varrho(\{x \in \mathcal{X} \,\colon\, Tx = x\}) < 1\), then it holds that
    \begin{align}
        \sigma_{\mathrm{ap}}(T) \cap \mathbb{S}^1\!\setminus\!\{1\} \neq \emptyset.
    \end{align}
\end{theorem}

Let us remark that for \emph{recurrent} linear operators \(T\), several forms of spectral restrictions have been obtained before. Costakis, Manoussos and Parissis demonstrated that every component of the spectrum \(\sigma(T)\) has to intersect the unit circle \(\mathbb{S}^1\), cf. \cite[Proposition 2.11]{Costakis2014Recurrent}. We remark that the same proof technique via Riesz projections can already be used to show that \(\sigma(T) \cap \mathbb{S}^1 \neq \emptyset\), but it is probably not easily adapted to show the same for the approximate point spectrum \(\sigma_{\mathrm{ap}}(T)\) instead of \(\sigma(T)\), and it cannot be used at all to establish \(\sigma(T) \cap \mathbb{S}^1\setminus\!\{1\} \neq \emptyset\), since the relevant Riesz projections are generally not available in this case.
The fact that isolated points of the spectrum \(\sigma(T)\) for \emph{upper-frequently recurrent} linear operators \(T\) have to lie on the unit circle \(\mathbb{S}^1\) was obtained, using techniques of Shkarin \cite{Shkarin2009}, by Bonilla, Grosse-Erdmann, López-Martínez and Peris, cf. \cite[Theorem 4.9]{bonilla2022frequently}. Not unrelatedly, Hedlund \cite[Theorem 1]{hedlund1971expansive} characterized \emph{uniform expansivity} of invertible bounded linear operators \(T\) by absence of approximate point spectrum \(\sigma_{\mathrm{ap}}(T)\) on the unit circle \(\mathbb{S}^1\).

Our Theorem~\ref{theorem:spectrum_on_circle_minus_1_from_nonfixing_invariant_law} is distinct in nature, since we assume only the existence of a non-trivial \(T\)-invariant probability measure \(\varrho\), which does not need to make \(T\) recurrent on the whole Banach space \(\mathcal{X}\). We specifically detect an approximate eigenvalue on \(\mathbb{S}^1 \setminus\{1\}\) whenever the measure \(\varrho\) is not concentrated on fixed points of \(T\).

Proofs for Theorem~\ref{theorem:example} and ~\ref{theorem:spectrum_on_circle_minus_1_from_nonfixing_invariant_law} are provided in the next section.

\subsection*{Outlook}
There are of course a lot of natural follow-up questions regarding sharpness of results, moment, etc. As I came to this question from a context where one actually would like to understand what all the possible invariant measures are, let me close this introduction with the following soft \textbf{questions}. Is it possible to give reasonable necessary and sufficient conditions on the involved operators or their spectral geometry that would more abstractly characterize the possibility of non-trivial invariant measures? In particular, is there a more abstract way to synthesize invariant non-fixing measure without presence of point-spectrum? Conversely, are there natural situations where the spectrum on the unit circle which is forced by existence of non-trivial invariant measures has to be point spectrum?
One concrete example I would like to understand is the operator \(\frac{\d}{\d x}(a(x) \frac{\d}{\d x})\) and the invariant measures for the semigroup it generates on power-weighted \(L^2\)-spaces or also Sobolev spaces (of negative regularity) on the real line \(\R\), where \(a\) is some very regular non-negative function that is allowed to grow at most quadratically, but also vanish in between.

\section{Proofs}

\subsection{The infinite-tape construction}\label{section:infinite_tape_construction}

In this section we will prove Theorem~\ref{theorem:example} by constructing the involved Hilbert space, bounded linear map, and probability measure in a hopefully somewhat transparent and pedagogical manner. A more concise version of the proof is given right at the end of the section.

The core idea consists of balancing membership of eigenvectors in the constructed Hilbert space versus almost-sure membership of an invariant process in this Hilbert space, by tuning the inner product as to favor the latter but excluding the former.

Now, first consider the very straight-forward construction of an invariant measure, in the case that \(T\) is injective but has an eigenvector \(x \in \mathcal{H}\) of \(T\) with eigenvalue \(\e^{i\alpha}\), \(\alpha \in [0, 2\pi)\), on the unit circle \(\mathbb{S}^1 = \{\theta \in \C \,\colon\, \vert\theta\vert = 1\}\). 
In all of the following, we will abuse notation and identify \(\theta \in \mathbb{S}^1\) with its unique argument in \([0, 2\pi)\); in particular we will write things like \(\theta + \alpha\) for \(\theta, \alpha \in \mathbb{S}^1\) and mean \(\theta \alpha \in \mathbb{S}^1\) or equivalently \(\theta + \alpha \mod 2\pi \in [0, 2\pi)\) for the arguments. 
Now, the distribution \(\varrho\) of the process
\begin{align*}
	X(\theta)
	:= \e^{i \theta} x,
\end{align*} with \(\theta\) sampled uniformly from \([0, 2 \pi)\), is \(T\)-invariant, because it has the property
\begin{align}\label{equation:pointwise_process_identity}
	T X(\theta)
	= X(\theta + \alpha \mod 2 \pi).
\end{align} 
But the property \eqref{equation:pointwise_process_identity} can also be realized in other ways, e.g.\ in the following. Consider as linear operator \(T\) the bilateral left-shift on \(\mathcal{K}^\Z\), for some fixed Hilbert space \(\mathcal{K}\), given by
\begin{align*}
	(Tx)_n 
	:= x_{n+1}, 
	\quad n \in \Z.
\end{align*}
If we define
\begin{align}\label{equation:definition_formal_process}
	[X(\theta)]_n
	:= f(\theta + n \alpha), \quad n \in \Z,
\end{align} for some function \(f \colon S^1 \to \mathcal{K}\), we get 
\begin{align*}
	[TX(\theta)]_n
	= [X(\theta)]_{n+1}
	= f(\theta + (n+1) \alpha)
	= f((\theta + \alpha) + n\alpha)
	= [X(\theta + \alpha)]_n,
\end{align*} i.e.\ \eqref{equation:pointwise_process_identity}.

Of course, this is just a formal construction yet, but we can implement it on a Hilbert space as follows.
\begin{definition}[Abstract infinite tape]\label{definition:abstract_infinite_tape}
    Let \(\mathcal{K}\) be some complex separable Hilbert space (the \emph{transverse Hilbert space}) and \(A = (A_n)_{n \in \mathbb{Z}}\) a bilateral sequence of bounded linear Operators \(A_n \colon \mathcal{K} \to \mathcal{K}\) with \(A_n \geq c_n\), for some \(c_n > 0\), \(n \in \Z\).
    We define the space
    \begin{align*}
    	\mathcal{H}_A
    	:= \{x = (x_n)_{n \in \Z} \in \mathcal{K}^\Z \,\colon\, \sum_{n \in \Z} \langle A_n x_n, x_n \rangle_{\mathcal{K}} < \infty\}
    \end{align*} and the inner product
    \begin{align*}
    	\langle x, y \rangle_{\mathcal{H}_A}
    	:= \sum_{n \in \Z} \langle A_n x_n, y_n \rangle_{\mathcal{K}}.
    \end{align*}
\end{definition}

Indeed, this gives a well-defined Hilbert space.
\begin{lemma}
	\((\mathcal{H}_A, \langle \cdot, \cdot \rangle_{\mathcal{H}_A})\) is a separable complex Hilbert space.
\end{lemma}

Also, the bilateral left-shift \(T\) can be made into a bounded linear map on \((\mathcal{H}_A, \langle \cdot, \cdot \rangle_{\mathcal{H}_A})\), if the family \(A\) is chosen correctly.
\begin{lemma}\label{lemma:conditions_for_bilat_left_shift_bounded_boundedly_invertible}
	Assume that there is some universal \(C > 0\) such that
	\begin{align}\label{equation:comparability_of_neighbouring_operators}
		C^{-1} A_n \leq A_{n+1} \leq C A_n, \quad n \in \Z.
	\end{align}
	Then, \(T\) is boundedly invertible bounded linear operator on \(\mathcal{H}_A\).
\end{lemma}

Now, let us think about how (formal) eigenvectors of \(T\) have to look like, so that we can later avoid having them as elements of \((\mathcal{H}_A, \langle \cdot, \cdot \rangle_{\mathcal{H}_A})\) by tuning \(A\) accordingly.
\begin{lemma}
    Suppose \(x \in \mathcal{K}^{\Z} \setminus\{0\}\) is such that \(T x = \lambda x\) for some \(\lambda \in \C\setminus\{0\}\).
    Then, \(x_n = \lambda^n x_0\) for all \(n \in \Z\).
\end{lemma}

On the other hand, we want to have that the (formal) process \(X\) defined in \eqref{equation:definition_formal_process} is almost-surely an element of \(\mathcal{H}_A\). This competition is something we have to balance later.

Let us first, try to think of simple operators \(A_n\) on \(\mathcal{K}\) that we can control easily enough to have the right properties.
The simplest thing that comes to mind are of course rank-one operators, i.e.\ given a bilateral sequence \((v_n)_{n \in \Z}\) of vectors in \(\mathcal{K}\) with \(\Vert v_n \Vert_{\mathcal{K}} \leq 1\), \(n \in \Z\), we can try and define
\begin{align*}
    A_n := \langle \cdot, v_n \rangle_{\mathcal{K}} \, v_n.
\end{align*}
Unfortunately, thinking of a vector that is orthogonal to \(v_n\) but not \(v_{n+1}\), we see that this will never give the comparability of neighbours \eqref{equation:comparability_of_neighbouring_operators} if \((v_n)_{n \in \Z}\) is not constant. Hence, we think of the exponential smearing
\begin{align*}
    A_n := \sum_{m \in \Z} \eta(n-m) \,\langle \cdot, v_m \rangle_{\mathcal{K}} \, v_m
\end{align*} with \(\eta(n) := \e^{- \vert n \vert}\), \(n \in \Z\), that clearly has the comparability of neighbours \eqref{equation:comparability_of_neighbouring_operators} because \(\frac{\eta(n+1)}{\eta(n)} \sim 1\).
Still, there is the problem that \(A_n\) may not be strictly positive and hence not boundedly invertible. To remedy this, we add some \emph{background noise} and finally set
\begin{align}\label{equation:fitting_definition_A_n}
    A_n := A_n^{\mathrm{cmp}} + b_n,
\end{align} \(b_n > 0\), with the compact part
\begin{align}
    A_n^{\mathrm{cmp}}
    := \sum_{m \in \Z} \eta(n-m) \,\langle \cdot, v_m \rangle_{\mathcal{K}} \, v_m
\end{align} from above, choosing the bilateral sequence \((v_n)_{n \in \Z}\) of vectors in \(\mathcal{K}\) with \(\Vert v_n \Vert_{\mathcal{K}} \leq 1\), \(n \in \Z\), and the bilateral sequence of strictly positive numbers \((b_n)_{n \in \Z}\) later.

We can now use the simple form of \(T\)-eigenvectors to give the following simple lower bound of their \(\mathcal{H}_A\)-norms.
\begin{lemma}
    Suppose \(x \in \mathcal{K}^\Z \setminus\{0\}\) is such that \(T x = \lambda x\) for some \(\lambda \in \C\setminus\{0\}\).
    Then,
    \begin{align}\label{equation:lower_bound_H_A_norm_eigenvectors}
        \Vert x \Vert_{\mathcal{H}_A}^2
        \gtrsim \sum_{n \in \Z} \vert\lambda\vert^{2n} \vert \langle x_0,  v_n \rangle_{\mathcal{K}}\vert^2
        \,+\, \Vert x_0\Vert_{\mathcal{K}}^2 \sum_{n \in \Z} b_n \vert\lambda\vert^{2n}.
    \end{align}
\end{lemma}

Similarly, it is easy to abstractly bound the \(\mathcal{H}_A\)-norms for the (formal) process \(X(\theta)\) from \eqref{equation:definition_formal_process} as follows.
\begin{lemma}
    We have
    \begin{align}\label{equation:abstract_bounds_H_A_norm_of_process}
        \Vert X(\theta) \Vert_{\mathcal{H}_A}^2
        \sim \sum_{n \in \Z} \eta(n) \Big( \sum_{m \in \Z} \vert \langle f(\theta + (n+m) \alpha), v_m \rangle_{\mathcal{K}} \vert^2 \Big) + \sum_{n \in \Z} b_n \Vert f(\theta + \alpha n) \Vert_{\mathcal{K}}^2.
    \end{align}
\end{lemma}

\subsubsection{Choosing a good transverse Hilbert space \(\mathcal{K}\) and operators \(A_n\)}

Recall that we wanted to choose \((A_n)_{n \in \Z}\), in turn \((v_n)_{n \in \Z}\) and \((b_n)_{n \in \Z}\), such that simultaneously eigenvectors of \(T\) can not belong to \(\mathcal{H}_A\) and \(X(\theta)\) is \(\d\theta\)-almost-surely an element of \(\mathcal{H}_A\).

We can already read off \eqref{equation:lower_bound_H_A_norm_eigenvectors} that it would be enough to exclude eigenvectors of \(T\) from \(\mathcal{H}_A\) by
\begin{enumerate}
    \item Choosing \(b_n\) with strictly subexponential decay, e.g.\ 
        \begin{align}\label{equation:choice_subexponential_background_noise}
            b_n := \e^{-\sqrt{1 + \vert n \vert}},
        \end{align} to exclude eigenvalues \(\lambda\) with \(\vert\lambda\vert \neq 1\),
    \item choosing \((v_n)_{n \in \Z}\) as \emph{anti-frame} with the property that
        \begin{align}\label{equation:anti_frame_property}
            \sum_{n \in \Z} \vert \langle x_0,  v_n \rangle_{\mathcal{K}}\vert^2
            = \infty
        \end{align} for all \(x_0 \in \mathcal{K}\setminus\{0\}\), to exclude eigenvalues \(\lambda\) with \(\vert\lambda\vert = 1\).
\end{enumerate} 
This is summarized in
\begin{lemma}\label{lemma:abstract_no_eigenvectors_in_H_A}
    Suppose that \((A_n)_{n \in \Z}\) is given as in \eqref{equation:fitting_definition_A_n}, where \((b_n)_{n \in \Z}\) is chosen as in \eqref{equation:choice_subexponential_background_noise} and \((v_n)_{n \in \Z}\) are such that the anti-frame property \eqref{equation:anti_frame_property} holds.
    Then \(T\) doesn't have any eigenvectors in \(\mathcal{H}_A\).
\end{lemma}

We will comment more on specific choices of the anti-frame later. To ensure \(\mathrm{d}\theta\)-almost-sure membership of \(X(\theta)\) in \(\mathcal{H}_A\), we already gave the computation \eqref{equation:abstract_bounds_H_A_norm_of_process}, but now we have to specify the transverse Hilbert space \(\mathcal{K}\) and the map \(f \colon \mathbb{S}^1 \to \mathcal{K}\) which are vital for that.

One natural choice of the transverse Hilbert space \(\mathcal{K}\) together with a map \(f \colon \mathbb{S}^1 \to \mathcal{K}\) is the well-known Sobolev space \(H^1(\mathbb{S}^1 \to \mathbb{C})\) of weakly differentiable functions \(x \colon \mathbb{S}^1 \to \mathbb{C}\) with the inner product
\begin{align}
    \langle x, y \rangle_{\mathcal{K}}
    := \int_{\mathbb{S}^1} \, \big[x(\theta) \overline{y}(\theta) + x'(\theta) \overline{y'}(\theta)  \big]\, \mathrm{d}\theta,
\end{align} because it is a reproducing kernel Hilbert space with reproducing kernel
\begin{align}
    f(\theta) 
    := k_\theta
    = \frac{1}{2\pi} \sum_{l \in \Z} \frac{\e^{i l(\cdot - \theta)}}{1 + l^2}.
\end{align}

Let us already remark here, that we will be able to resolve the competition between having no eigenvectors of \(T\) in \(\mathcal{H}_A\) and having \(\d\theta\)-almost-sure membership of \(X(\theta)\) in \(\mathcal{H}_A\) by noticing the following (non-)invariance. The anti-frame property \eqref{equation:anti_frame_property} does not really depend on the \emph{geometry} of the infinite-tape, i.e.\ \emph{where} on the line \(\Z\) we place each member of \((v_n)_{n \in \Z}\), but the abstract bound \eqref{equation:abstract_bounds_H_A_norm_of_process} now gives in our concrete setting
\begin{align}\label{equation:concrete_bounds_H_A_norm_of_process}
    \Vert X(\theta) \Vert_{\mathcal{H}_A}^2
    \lesssim 1 + \sum_{n \in \Z} \e^{-\vert n\vert} \Big( \sum_{m \in \Z} \vert v_m(\theta + \alpha (n+m)) \vert^2 \Big).
\end{align}

Now, we will construct the anti-frame \((v_n)_{n \in \Z}\) from the following family of \emph{hat functions} in \(\mathcal{K}\), which is flexible enough and also easy to do computations with.
\begin{definition}[Family of hat functions]
    Let \(r \in \mathbb{N}\) and consider the uniform mesh \(\mathfrak{M}_r := \{2 \pi k h_r \,\colon\, 0 \leq k < M_r\}\) of \([0, 2 \pi)\), with \(h_r := 8^{-r}\), \(M_r := h_r^{-1}\). Denote by \(V_r\) the space of continuous (periodic) \(H^1(\mathbb{S}^1 \to \mathbb{C})\)-functions which are affine linear on every interval \([2\pi k h_r, 2 \pi (k+1) h_r)\) associated with the mesh \(\mathfrak{M}_r\).
    A simple basis of \(V_r\) are the (periodic) affine linear \emph{hat functions} \(\phi_{r, k}\) associated with the mesh \(\mathfrak{M}_r\), which have \(\phi_{r, k}(2\pi j h_r) = \mathbf{1}_{j = k}\), \(0 \leq k < M_r\). 
    Further, we consider the normalized
    \begin{align}
        u_{r, k} := \frac{\phi_{r, k}}{\Vert \phi_{r, k} \Vert_{\mathcal{K}}}
    \end{align} for \(r \in \mathbb{N}\) and \(0 \leq k < M_r\).
\end{definition}

We give the following simple anti-frame-bound for \((u_{r, k})_{0 \leq < M_r}\) with fixed mesh-size parameter \(r \in \N\), which follows easily by explicit computations, using the triangular shape of the hat functions \(\phi_{r, k}\) and that they form a basis of \(V_r\) defined above.
\begin{lemma}
    Let \(x \in \mathcal{K}\) and \(P_{V_r}\) be the orthogonal projection on \(V_r\), \(r \in \mathbb{N}\).
    Then,
    \begin{align}\label{equation:anti_frame_bound}
        \sum_{k = 0}^{M_r - 1} \vert \langle x, u_{r, k} \rangle_{\mathcal{K}} \vert^2
        \geq \frac{h_r^2}{3} \Vert P_{V_r} x\Vert_{\mathcal{K}}^2.
    \end{align}
\end{lemma}

It is now easy to see from \eqref{equation:anti_frame_bound}, that any family \((v_n)_{n \in \Z}\) that contains at least \(N_r := \lceil h_r^{-2} \rceil\)-many copies of each \(u_{k, r}\) will have the anti-frame property \eqref{equation:anti_frame_property}.
It remains to see that we can arrange \((v_n)_{n \in \Z}\) on the line \(\Z\) in such a way that the upper bound \eqref{equation:concrete_bounds_H_A_norm_of_process} will be finite \(\d\theta\)-a.s.
\begin{lemma}\label{lemma:fitting_arrangement_of_anti_frame_on_line_Z}
    Assume that \(\alpha \in \mathbb{R}\) is chosen such that \(\frac{\alpha}{2\pi}\) is irrational.
    Let \(\mathcal{I} := \{(r, k, l) \,\colon\, r \in \N, 0 \leq k < M_r, 1 \leq l \leq N_r\}\). Then there exists an injection \(j \colon \mathcal{I} \to \mathbb{Z}\) such that the family \((v_n)_{n \in \Z}\) given by
    \begin{align}\label{equation:fitting_choice_anti_frame_v_n}
        v_n := \begin{cases}
            u_{r, k}, &\text{if there is some } l \text{ with } (r, k, l) = j^{-1}(n)  , \\
            0, & \text{otherwise},
        \end{cases}
    \end{align} has the property that it \(\d\theta\)-almost-surely holds
    \begin{align}\label{equation:finiteness_phase_load}
         \sum_{n \in \Z} \e^{-\vert n\vert} \Big( \sum_{m \in \Z} \vert v_m(\theta + \alpha (n+m)) \vert^2 \Big)
         < \infty.
    \end{align}
\end{lemma}
\begin{proof}
    We have to arrange \((v_n)_{n \in \Z}\) such that \eqref{equation:finiteness_phase_load} holds \(\d\theta\)-almost-surely. Noticing that the distribution of \(\sum_{m \in \Z} \vert v_m(\theta + \alpha (n+m)) \vert^2\) under \(\d\theta\) is independent of \(n\), it suffices, by the Borel-Cantelli lemma and the layer cake representation of expected values, to enforce
    \begin{align}\label{equation:log_moment_phase_load}
        \int_{\mathbb{S}^1} \, \log\Big(1  +   \sum_{m \in \Z} \vert v_m(\theta 
        + \alpha m ) \vert^2 \Big) \, \d\theta
        < \infty.
    \end{align} 
    This can be done by the following pointwise bounds.
    On the one hand, one can easily compute for the normalized hat functions \(u_{r, k}\), that
    \begin{align}
        \Vert u_{r, k} \Vert_\infty^2
        \leq \pi h_r.
    \end{align}
    On the other hand, we can decompose (up to countably many points) the unit circle \(\mathbb{S}^1 \cong [0, 2\pi)\) into disjoint arcs \(C_r = (2 \pi \cdot 2^{-r}, 2 \pi \cdot 2^{-r + 1})\) of length \(2 \pi \cdot 2^{-r}\), and then construct an injection \(j \colon \mathcal{I} \to \Z\) such that
    \begin{align}
        \mathrm{supp}(u_{r, k})
        \subset C_r + \alpha j(r, k, l)
    \end{align} for all \((r, k, l) \in \mathcal{I}\). This is possible because \(\{ j\alpha \colon j \in \Z\}\) meets every non-empty arc infinitely often by the irrationality of \(\alpha\), so that we can construct \(j\) greedily.

    Hence, we constructed \(j\) such that with \(v_n\) from \eqref{equation:fitting_choice_anti_frame_v_n}, that for fixed \(r \in \N\) and \(\theta \in C_r\) only the \(N_r\)-many copies of \((u_{r, k})_{0 \leq k < M_r}\) in \((v_n)_{n \in \Z}\) can contribute to the sum
    \begin{align}
        \sum_{m \in \Z} \vert v_m(\theta + \alpha m ) \vert^2
        \leq N_r M_r \sup_{0 \leq k < M_r} \Vert u_{r, k} \Vert_\infty^2
        \lesssim \pi N_r M_r h_r
        = \pi N_r \sim 64^r.
    \end{align} 
    Finally, it follows \eqref{equation:log_moment_phase_load} for our choice of \((v_n)_{n \in \Z}\) by
    \begin{align}
        \int_{\mathbb{S}^1} \, \log\Big(1  +   \sum_{m \in \Z} \vert v_m(\theta 
        + \alpha m ) \vert^2 \Big) \, \d\theta
        \leq \sum_{r \geq 1} \vert C_r\vert \log(1 + \pi N_r)
        \lesssim \sum_{r \geq 1} r 2^{-r} 
        < \infty.
    \end{align}
\end{proof}

From the preceding lemmas, we can easily construct the sought-after invariant measure on \(\mathcal{H}_A\).
\begin{lemma}\label{lemma:the_invariant_measure}
    Suppose that \((A_n)_{n \in \Z}\) is given as in \eqref{equation:fitting_definition_A_n}, where \((b_n)_{n \in \Z}\) is chosen as in \eqref{equation:choice_subexponential_background_noise} and \((v_n)_{n \in \Z}\) is arranged as in \eqref{equation:fitting_choice_anti_frame_v_n}.
    Then, there is a non-trivial \(T\)-invariant probability measure \(\varrho \neq \delta_0\) on \(\mathcal{H}_A\) (equipped with the Borel-\(\sigma\)-algebra).
\end{lemma}
\begin{proof}
    Consider, as before, the function \(X \colon \mathbb{S}^1 \to \mathcal{K}^\Z\) given by
    \begin{align}
        [X(\theta)]_n
        = k_{\theta + \alpha n},
    \end{align} and also its truncations, for \(N \in \N\),
    \begin{align}
        [X^{(N)}(\theta)]_n
        := k_{\theta + \alpha n} \mathbf{1}_{\vert n\vert \leq N}, \quad n \in \Z,
    \end{align} which are continuous as maps \(\mathbb{S}^1 \to \mathcal{H}_A\).
    
    We can construct a strongly measurable version \(\widetilde{X}\) of \(X\) with values in \(\mathcal{H}_A\) by first defining 
    \begin{align}
        \Omega^*
        := \{\theta \in \mathbb{S}^1 \,\colon\, \Vert X(\theta) \Vert_{\mathcal{H}_A}^{2} = \sup_{N} \Vert X^{(N)}(\theta)\Vert_{\mathcal{H}_A}^2 < \infty \}
    \end{align} and then
    \begin{align}
        \widetilde{X}
        := X \mathbf{1}_{\Omega^*}.
    \end{align} 
    Indeed, by boundedness and bounded invertibility of \(T\), the set \(\Omega^*\) is rotation-invariant, so that we have \(T\widetilde{X}(\theta) = \widetilde{X}(\theta + \alpha)\) for all \(\theta \in \mathbb{S}^1\). Also, \(\Omega^{*}\) is a set of full \(\d\theta\)-measure by \eqref{equation:concrete_bounds_H_A_norm_of_process} and Lemma~\ref{lemma:fitting_arrangement_of_anti_frame_on_line_Z}.
    Putting everything together, the statement of this lemma follows by letting
    \begin{align}
        \varrho 
        := \frac{1}{\vert \mathbb{S}^1 \vert}\int_{\mathbb{S}^1} \, \delta_{\widetilde{X}(\theta)} \, \d\theta.
    \end{align}
\end{proof}

Let us finally gather concisely everything we introduced step-wise in a
\begin{proof}[Proof of Theorem~\ref{theorem:example}]
    We consider the Hilbert space \(\mathcal{H}_A\) as defined in Definition~\ref{definition:abstract_infinite_tape}. Herein, the operators \(A_n\) are chosen as in \eqref{equation:fitting_definition_A_n} with the background noise \(b_n\) given in \eqref{equation:choice_subexponential_background_noise}, and the compact part defined via the anti-frame \((v_n)_{n \in \Z}\) from \eqref{equation:fitting_choice_anti_frame_v_n}. The bilateral left-shift \(T\) is then a boundedly invertible linear operator on \(\mathcal{H}_A\) by Lemma~\ref{lemma:conditions_for_bilat_left_shift_bounded_boundedly_invertible}.
    A non-trivial \(T\)-invariant measure \(\varrho\) on \(\mathcal{H}_A\) is given in Lemma~\ref{lemma:the_invariant_measure}. Finally, Lemma~\ref{lemma:abstract_no_eigenvectors_in_H_A} shows together with the small observation right before Lemma~\ref{lemma:fitting_arrangement_of_anti_frame_on_line_Z} that \(T\) cannot have any eigenvectors in \(\mathcal{H}_A\).
\end{proof}

\subsection{Non-trivial invariant measures and spectrum on \(\mathbb{S}^1\)}

Here we give a small
\begin{proof}[Proof of Theorem~\ref{theorem:spectrum_on_circle_minus_1_from_nonfixing_invariant_law}]
    By the topological Poincaré recurrence theorem, cf. \cite[Theorem 3.3]{furstenberg1981recurrence}, we have for \(\varrho\)-almost all \(x \in \mathcal{X}\) that
    \begin{align}\label{equation:topological_poincare_recurrence}
        (T^{N_m} - 1) x 
        \xrightarrow[m \to \infty]{\Vert \cdot\Vert_{\mathcal{X}}}
        0
    \end{align} for some increasing subsequence \((N_m)_{m \in \N}\) of \(\N\).
    We will extract an approximate eigenvalue on the unit circle \(\mathbb{S}^1\) from that. The idea is to first approximate the approximate eigenvalue and look at angles which are rational multiples \(2 \pi \frac{k}{N}\) of \(2 \pi\) only, and compute formally at first
    \begin{align}\label{equation:computation_approx_approx_eigvecs}
        [T - \e^{2 \pi i \frac{k}{N}}]^{-1} \,(T^N - 1)
        = \e^{-2\pi i \frac{k}{N}} \frac{(\e^{-2\pi i \frac{k}{N}} T)^N - 1}{\e^{-2\pi i \frac{k}{N}} T - 1}
        = \e^{-2\pi i \frac{k}{N}} \sum_{n = 0}^{N-1} \e^{-2\pi i n \frac{k}{N}} T^n.
    \end{align}
    In other words, for \(u_{N, k} := \sum_{n = 0}^{N-1} \e^{-2\pi i n \frac{k}{N}} T^n x\), we rigorously get by telescoping the identity
    \begin{align}\label{equation:identity_conversion_approx_fixed_point_approx_eigvalue_on_unit_circle}
        [T - \e^{2 \pi i \frac{k}{N}}] u_{N, k}
        = \e^{2\pi i \frac{k}{N}} (T^N - 1)x.
    \end{align}
    
    In particular, from \eqref{equation:identity_conversion_approx_fixed_point_approx_eigvalue_on_unit_circle} together with \eqref{equation:topological_poincare_recurrence}, we can conclude for any choice of a sequence \((k_m)_{m \in \mathbb{N}}\) in \(\Z\), that
    \begin{align}
        \Vert (T - \e^{i \theta_{m}}) \widetilde{x}_{m} \Vert_{\mathcal{X}}
        \xrightarrow[m \to \infty]{} 0.
    \end{align} where \(\theta_m := 2 \pi \frac{k_m}{N_m}\) and \(\widetilde{x}_m := u_{N_m, k_m}\).
    
    Since \([0, 2\pi]\) is compact, we can w.l.o.g. assume that \(\theta_m \xrightarrow[m \to \infty]{} \theta\) for some \(\theta \in [0, 2\pi)\).
    Normalize \(x_m := \frac{\widetilde{x}_m}{\Vert \widetilde{x}_m\Vert_{\mathcal{X}}}\). It then follows that
    \begin{align}
        \Vert (T - \e^{i \theta}) x_{m} \Vert_{\mathcal{X}}
        \leq \frac{1}{\Vert \widetilde{x}_m \Vert} \Vert (T - \e^{i \theta_m}) \widetilde{x}_{m} \Vert_{\mathcal{X}} + \vert \e^{i \theta} - \e^{i \theta_m} \vert
        \xrightarrow[m \to \infty]{} 0,
    \end{align} if we manage to choose \((k_m)_{m \in N}\) in such a way that 
    \begin{align}\label{equation:lower_bound_norms_approx_approx_eigvecs}
        \inf_{m \in \N} \Vert \widetilde{x}_m \Vert \geq \widetilde{\delta} > 0
    \end{align} In this case, it is \(\e^{i \theta} \in \sigma_{\mathrm{ap}}(T)\).
    
    Furthermore, when \(\varrho(\{x \in \mathcal{X} \,\colon\, Tx = x\}) = 1\), it is already clear from \(\varrho \neq \delta_0\) that \(1\) is an eigenvalue of \(T\), so that \(1 \in \sigma_{\mathrm{ap}}(T) \cap \mathbb{S}^1\). Consequently, we now assume  \(\varrho(\{x \in \mathcal{X} \,\colon\, Tx = x\}) < 1\), w.l.o.g. by conditioning even \(\varrho(\{x \in \mathcal{X} \,\colon\, Tx = x\}) = 0\), and have to additionally show that then \(\e^{i \theta} \neq 1\).

    Let us now start with choosing \(x \in \mathcal{X}\) and \((k_m)_{m \in \N}\) such that \eqref{equation:lower_bound_norms_approx_approx_eigvecs} holds, for which we will have to in some form establish that there has to be some \(k\) such that there cannot be too much cancellation among the summands \(\mathrm{e}^{-2 \pi i n \frac{k}{N}} T^n x\) on the r.h.s. of \eqref{equation:computation_approx_approx_eigvecs}, due to the fact that \(Tx \neq x\) and that \(T\) has some recurrence properties.
    Now, first, select some \(y \in \mathrm{supp}(\varrho)\) such that \(T y \neq y\). The Hahn-Banach theorem provides us a separating linear functional \(\varphi \in \mathcal{X}^{*}\), \(\Vert \varphi \Vert_{\mathcal{X}^*} = 1\), obeying
    \begin{align}
        \delta 
        := \mathrm{Re} [ \varphi(T y - y) ] > 0.
    \end{align}
    If we put \(c := \frac{\mathrm{Re}[\varphi(y)] + \mathrm{Re}[\varphi(Ty)]}{2}\), there is, by boundedness of \(T\), some open neighborhood \(U\) around \(y\) such that
    \begin{align}\label{equation:separation_of_z_Tz_by_linear_functional}
        \mathrm{Re}[\varphi(z)] - c < -\frac{\delta}{4}, \quad \mathrm{Re}[\varphi(Tz)] - c > \frac{\delta}{4}, \quad z \in U.
    \end{align}
    Since \(\varrho(U) > 0\), there is some \emph{frequently recurrent} \(x \in U\), cf. \cite[Lemma 3.1]{GrivauxMartinez2023}, which in particular satisfies
    \begin{align}\label{equation:frequent_recurrence_in_U}
        p := \liminf_{N \to \infty} \frac{1}{N} \vert \{1 \leq n \leq N \,\colon\, T^n x \in U \} \vert > 0.
    \end{align} We fix \(x\) so that simultaneously the topological Poincaré recurrence \cref{equation:topological_poincare_recurrence} holds.

    Let
    \begin{align}
        \widehat{\Phi}_N(k)
        := \sum_{n = 0}^{N-1} \e^{-2 \pi i n \frac{k}{N}}  (\mathrm{Re}[ \varphi(T^n x) ] - c).
    \end{align}
    For \(k \notin N\Z\) it holds that
    \begin{align}
        \widehat{\Phi}_N(k)
        = \sum_{n = 0}^{N-1} \e^{-2 \pi i n \frac{k}{N}}  \, \mathrm{Re}[ \varphi(T^n x) ] 
        = \frac{1}{2}\Big(\varphi(u_{N, k}) + \overline{\varphi(u_{N, -k})} \Big),
    \end{align} so that it would suffice for our purpose of finding \(k = k(N)\) such that \(\Vert u_{N, k} \Vert_{\mathcal{X}} \geq \frac{\delta}{4}\), recalling that \(\Vert\varphi\Vert_{\mathcal{X}*} = 1\), to find such a \(k\) (replacing it if needed by \(-k\)) that \(\vert \widehat{\Phi}_N(k) \vert \geq \frac{\delta}{4}\).
    This can be done by Lemma~\ref{lemma:fourier_lemma}, considering there \(b_n := \mathrm{Re}[ \varphi(T^n x) ] - c\). We can apply the same lemma to ensure that uniformly
    \begin{align}
        \vert \e^{i \theta_m} - 1\vert
        = \vert \e^{2 \pi i \frac{k_m}{N_m}} - 1\vert
        = 2 \sin\Big(\pi \frac{\min_{j \in \Z} \vert k_m - jN_m\vert }{N_m} \Big)
        \geq \sin(\pi p /4) > 0,
    \end{align} because \(b_n = \mathrm{Re}[ \varphi(T^n x) ] - c\) has at least \(\lfloor \frac{p N_m}{2} \rfloor\)-many \(\frac{\delta}{4}\)-upcrossings by the frequent visits \eqref{equation:frequent_recurrence_in_U} of \(T^n\) to \(U\) and separation  \eqref{equation:separation_of_z_Tz_by_linear_functional} of \(z \in U\) and \(Tz\).
\end{proof}

We needed the following lemma which formalizes the intuition that some finite sequence with enough upcrossings cannot be concentrated on low-frequency Fourier modes only.
\begin{lemma}\label{lemma:fourier_lemma}
    Let \(N \geq 2\) and let \(b_0, \dots, b_{N-1} \in \R\) have, for fixed \(\eta > 0\), at least \(r \in \N\) many \(\eta\)-upcrossings
    \begin{align}
        b_{n_\ell} < -\eta, \quad b_{n_\ell +1} > \eta, 
    \end{align} for \(0 \leq n_{1} <\dots < n_r \leq N -2\).
    Define
    \begin{align}
        \widehat{b}(k)
        = \sum_{n = 0}^{N-1} \e^{-2\pi i n \frac{k}{N}} b_n, \quad k \in \Z.
    \end{align}
    Then there is some \(0 \leq k < N\) such that
    \begin{align}\label{equation:higher_fourier_mode_bound}
        \vert \widehat{b}(k) \vert \geq \eta, \quad \min_{j \in \Z} \vert k - jN \vert \geq r.
    \end{align}
\end{lemma}
\begin{proof}
    Denote \(\vert k \vert_N = \min_{j \in \Z} \vert k - jN \vert \).
    Suppose there is no \(0 \leq k < N\) such that \eqref{equation:higher_fourier_mode_bound} holds, which implies
    \begin{align}\label{equation:not_much_high_frequency_mass}
        \sum_{k = 0}^{N-1} \vert \widehat{b}(k) \vert \mathbf{1}_{\vert k\vert_N \geq r}
        < \eta N
    \end{align} holds. Then, the low-frequency restriction of the inverse Fourier transform
    \begin{align}
        P(t)
        = \frac{1}{N} \sum_{k = -r+1}^{r-1} \e^{2 \pi i t \frac{k}{N}} \widehat{b}(k), \quad t \in \R,
    \end{align} is a real trigonometric polynomial of degree \(r - 1\) and
    \begin{align}
        \vert P(n) - b_n \vert
        < \eta, \quad 0 \leq n < N.
    \end{align} Hence, the \(r\)-many \(\eta\)-upcrossings imply that \(P\) has to have \((2r-1)\)-many zeros in \([0, N)\). But \(P\), as a non-zero trigonometric polynomial of degree \(r-1\), can have at most \(2 (r-1) < 2r - 1\) zeros. It follows that \eqref{equation:not_much_high_frequency_mass} cannot hold and there is a \(0 \leq k < N\) with \eqref{equation:higher_fourier_mode_bound}.
\end{proof}

\subsection*{Acknowledgments} 
I thank Jonas Köppl for pleasant discussions on a related joint research project.

\subsection*{Disclosure of AI use}
In the context of another research project, I already came up with the core idea of balancing membership of eigenvectors in Hilbert space vs. almost-sure membership of an invariant process in this Hilbert space to understand possible invariant measures. This lead to me conjecturing Theorem~\ref{theorem:example} with this very abstract ansatz.
But, the concrete counterexample and proof were then solely generated by \texttt{ChatGPT 5.6 Sol}, and only clarified and written down more clearly by me.
In the cases I had in mind (just as here), there are still some generalized eigenvectors and this should somehow mirror in the approximate point spectrum as in the other Theorem~\ref{theorem:spectrum_on_circle_minus_1_from_nonfixing_invariant_law} here. Again, apart from the first part of this theorem, the concrete proof idea was solely generated by \texttt{ChatGPT 5.6 Sol}.
I take full responsibility for the correctness of the arguments presented in this note.

\bibliographystyle{alpha}
\bibliography{references}

@article{GrivauxMartinez2023,
    title = {Recurrence properties for linear dynamical systems: An approach via invariant measures},
    journal = {Journal de Mathématiques Pures et Appliquées},
    volume = {169},
    pages = {155-188},
    year = {2023},
    author = {Sophie Grivaux and Antoni López-Martínez},
}

@article{grivaux2021linear,
  title={Linear dynamical systems on {H}ilbert spaces: typical properties and explicit examples},
  author={Grivaux, Sophie and Matheron, {\'E}tienne and Menet, Quentin},
  journal={Memoirs of the American Mathematical Society},
  volume={269},
  number={1315},
  year={2021},
  publisher={American Mathematical Society},
}

@book{furstenberg1981recurrence,
  title     = {Recurrence in Ergodic Theory and Combinatorial Number Theory},
  author    = {Furstenberg, Harry},
  year      = {1981},
  publisher = {Princeton University Press},
  address   = {Princeton, NJ},
}

@article{Flytzanis1995,
  author = {Flytzanis, Elias},
  title = {Unimodular eigenvalues and invariant measures for linear operators},
  journal = {Monatshefte f{\"u}r Mathematik},
  volume = {119},
  number = {4},
  pages = {267--273},
  year = {1995},
  publisher = {Springer},
}

@article{bonilla2022frequently,
  title={Frequently recurrent operators},
  author={Bonilla, A. and Grosse-Erdmann, K.-G. and L{\'o}pez-Mart{\'i}nez, A. and Peris, A.},
  journal={Journal of Functional Analysis},
  volume={283},
  number={12},
  pages={109713},
  year={2022},
  publisher={Elsevier},
}

@article{Costakis2014Recurrent,
  author  = {Costakis, George and Manoussos, Antonios and Parissis, Ioannis},
  title   = {Recurrent Linear Operators},
  journal = {Complex Analysis and Operator Theory},
  volume  = {8},
  pages   = {1601--1643},
  year    = {2014},
}

@article{Shkarin2009,
  author = {Shkarin, Stanislav},
  title = {On the spectrum of frequently hypercyclic operators},
  journal = {Proceedings of the American Mathematical Society},
  volume = {137},
  number = {1},
  pages = {123--134},
  year = {2009},
}

@article{hedlund1971expansive,
  author    = {Hedlund, James H.},
  title     = {Expansive Automorphisms of {B}anach Spaces, {II}},
  journal   = {Pacific Journal of Mathematics},
  volume    = {36},
  number    = {3},
  pages     = {671--675},
  year      = {1971},
  publisher = {Mathematical Sciences Publishers},
}

\end{document}